\documentclass[11pt,a4paper]{article}
\usepackage[top=2.5cm,bottom=2.5cm,left=2.2cm,right=2.2cm]{geometry}
\usepackage[T1]{fontenc}
\usepackage[utf8]{inputenc}
\usepackage{color}
\usepackage{graphicx}
\usepackage{amsfonts}
\usepackage{extarrows}
\usepackage{amsmath,amsthm,amssymb,color}
\usepackage{hyperref}
\usepackage{eepic}
\usepackage{lineno}
\usepackage{enumerate}	
\usepackage{paralist}
\usepackage{cite}
\usepackage{algorithm}
\usepackage{algorithmicx}
\usepackage{algpseudocode}
\usepackage{authblk}
\usepackage{mathtools}
\usepackage{pifont}
\usepackage[numbers,sort&compress]{natbib}

\usepackage{tikz}

\hypersetup
{
	colorlinks=true,
	linkcolor=blue,
	filecolor=blue,
	urlcolor=blue,
	citecolor=cyan,
}

\newtheorem{theorem}{Theorem}[section]

\newtheorem{lemma}[theorem]{Lemma}
\newtheorem{corollary}[theorem]{Corollary}

\newtheorem{observation}[theorem]{Observation}

\theoremstyle{definition}
\newtheorem{definition}[theorem]{Definition}

\newtheorem{case}{\indent Case}[section]

\newtheorem{subcase}{\indent Subcase}[case]

{\setlength{\leftmargini}{1.5\parindent}
 \begin{itemize}
 \setlength{\itemsep}{-1.1mm}}
{\end{itemize}}

\begin{document}
	\title{\bf An Ore-type condition for hamiltonicity in graphs}
	\author{\bf Chengli Li\footnote{Email: lichengli0130@126.com.}}
	\author{\bf Feng Liu\footnote{Email: liufeng0609@126.com (corresponding author).}}
	
	\affil{ Department of Mathematics,
		East China Normal University, Shanghai, 200241, China}
	\date{}
	\maketitle
\begin{abstract}
The bipartite-hole-number of a graph $G$, denoted as $\widetilde{\alpha}(G)$, is the minimum number $k$ such that there exist positive integers $s$ and $t$ with $s+t=k+1$ with the property that for any two disjoint sets $A,B\subseteq V(G)$ with $|A|=s$ and $|B|=t$, there is an edge between $A$ and $B$. In this paper, based on Ore-type conditions, we show that if a graph $G$ is 2-connected and the degree sum of any two nonadjacent vertices in $G$ is at least $ 2\widetilde{\alpha}(G)$, then $G$ is hamiltonian. Furthermore, we prove that if $G$ is 3-connected and the degree sum of any two nonadjacent vertices in $G$ is at least $ 2\widetilde{\alpha}(G)+1$, then $G$ is hamiltonian-connected.

			
\end{abstract}
	
\section{Introduction}
We consider finite simple  graphs, and use standard terminology and notations from \cite{Bondy2008, West1996} throughout this article. We denote by $V(G)$ and $E(G)$ the vertex set and edge set of a graph $G,$ respectively, and denote by $|G|$ and $e(G)$ the order and size of $G,$ respectively. For a vertex $x\in V(G)$ and a subgraph $H$ of $G$, $N_H(x)$ denotes  the set of neighbors of $x$ that are contained in $V(H)$. For a vertex subset $S\subseteq V(G)$, define $N_G(S)=\cup_{x\in S}N_G(x)\setminus S$ and $N_H(S)=N_G(S)\cap V(H)$. If $F$ is a subgraph of $G$, we write $N_F(H)$ for $N_F(V(H))$. We use $G[S]$ to denote the subgraph of $G$ induced by $S$, and let $G-S=G[V(G)\setminus S]$. Given two vertex subsets $S$ and $T$ of $G$, we denote by $[S,T]$ the set of edges having one endpoint in $S$ and the other in $T$ of $G$. For a positive integer $k$, the symbol $[k]$ used in this article represents the set $\{1,2,\ldots,k\}$. Furthermore, for integers $a$ and $b$ with $a\le b,$ we use $[a, \,b]$ to denote the set of those integers $c$ satisfying $a\le c\le b.$
The subscript $G$ will be omitted in all the notation above if no confusion may arise. 

Define $\sigma_2(G)=\min \{d_G(u)+d_G(v)~:~ u,v\in V(G)~\text{and}~ u\nsim v\}$ if $G$ is not a complete graph, and define $\sigma_2(G)=\infty$ otherwise.

For two distinct vertices $x$ and $y$ in $G$, an $(x, y)$-path is a path whose endpoints are $x$ and $y$.  Let $P$ be a path. We   use $P[u,v]$ to denote the subpath of $P$ between two  vertices $u$  and $v$. 

A Hamilton path in $G$ is a path containing every vertex of $G$.  A Hamilton  cycle in $G$ is a cycle containing every vertex of $G$. A graph $G$ is traceable if it contains a Hamilton path, and it is hamiltonian if it contains a Hamilton cycle. 
 
The classic  Dirac theorem from 1952 is as follows.
\begin{theorem}[Dirac \cite{Dirac1952}]
Let $G$ be a graph of order at least three. If $\delta(G)\geq \frac{n}{2}$, then $G$ is hamiltonian.
\end{theorem}
A lot of effort have been made by various people in generalization of Dirac’s theorem and this area is one of the core subjects in hamiltonian graph theory. For more information on some of these generalizations, we refer the reader to  \cite{Dirac1952,Fan1984,Faudree1989,Gould2014,Li2013,Ore1960}. The first important generalization was obtained by Ore in 1960.
\begin{theorem}[Ore \cite{Ore1960}]
Let $G$ be a graph of order at least three. If $\sigma_2(G)\geq n$, then $G$ is hamiltonian.
\end{theorem}
Dirac \cite{Dirac1952} and Ore \cite{Ore1960} laid the groundwork for hamiltonian graph theory. Results based on minimum degree are called Dirac-type, while those involving $\sigma_2(G)$ are known as Ore-type. The following notion of bipartite hole was introduced by 
McDiarmid and Yolov \cite{Mcdiarmid2017} in the study of Hamilton cycles.
\begin{definition}
An $(s,t)$-bipartite-hole in a graph $G$ consists of two disjoint sets of vertices, $S$ and $T$, with $|S| = s$ and $|T| = t$, such that $[S, T] = \emptyset$. The {\it bipartite-hole-number} of a graph $G$, denoted as $\widetilde{\alpha}(G)$, is the minimum number $k$ such that there exist positive integers $s$ and $t$ with $s + t = k + 1$, and such that $G$ does not contain an $(s,t)$-bipartite-hole.
\end{definition}

An equivalent definition of $\widetilde{\alpha}(G)$ is the maximum integer $r$ such that $G$ contains an $(s,t)$-bipartite-hole for every pair of nonnegative integers $s$ and $t$ with $s+t=r$.

In 2017, McDiarmid and Yolov \cite{Mcdiarmid2017} provided a sufficient condition for hamiltonicity  in terms of the minimum degree and the  bipartite-hole-number.
\begin{theorem}[McDiarmid-Yolov \cite{Mcdiarmid2017}]\label{dirac-hamiltonian}
Let $G$ be a graph of order at least three.  If $\delta(G)\ge \widetilde{\alpha}(G)$, then $G$ is hamiltonian.
\end{theorem}
A graph is called  hamiltonian-connected if between any 
two distinct vertices there is a Hamilton path. 
The following well-known theorem, established by Ore, provides the corresponding degree sum conditions for any graph to be hamiltonian-connected.
\begin{theorem}[Ore \cite{Ore1963}]
Let $G$ be a graph of order at least three.  If $\sigma_2(G)\geq n+1$, then $G$ is hamiltonian-connected.
\end{theorem}
In 2024, Zhou, Broersma, Wang and Lu provided a sufficient condition for hamiltonian connectedness based on the minimum degree and the bipartite-hole-number.
\begin{theorem}[Zhou-Broersma-Wang-Lu \cite{Zhou2024}]\label{dirac-hamiltonian-connected}

Let $G$ be a graph of order at least three.  If $\delta(G)\ge \widetilde{\alpha}(G)+1$, then $G$ is hamiltonian-connected.
\end{theorem}
There has been much recent work on the bipartite-hole-number. For more references, the reader may refer to \cite{Chen2022,Han2024,Draganic2024}.  Our first result is a sufficient condition for hamiltonicity in terms of $ \sigma_2(G) $ and the bipartite-hole-number.


\begin{theorem}\label{Theorem-ore-hamiltonian}
Let $G$  be a $2$-connected graph of order at least three. If $\sigma_2(G) \geq 2\widetilde{\alpha}(G)$, then $G$ is hamiltonian.
\end{theorem}
The condition that the graph is $2$‑connected in Theorem \ref{Theorem-ore-hamiltonian} is necessary. To see this, let $ G $ be the graph obtained by taking the disjoint union of $ K_a $ and $ K_b $ and then adding a single edge, where $ b \ge 3a + 4 $. Clearly, 
$\sigma_2(G) \ge 4a + 2 = 2\widetilde{\alpha}(G),$
yet $ G $ is not hamiltonian.

As an application of Theorem \ref{Theorem-ore-hamiltonian}, we have the following corollary.
\begin{corollary}\label{Corollary-ore-traceable}
Let $G$  be a connected graph. If $\sigma_2(G) \ge 2\widetilde{\alpha}(G)-2$, then $G$ is traceable. 
\end{corollary}

Our second result is a sufficient condition for hamiltonian connectedness in terms of $\sigma_2(G) $ and the bipartite-hole-number.

\begin{theorem}\label{Theorem-ore-hamiltonian-connected}
Let $G$  be a $3$-connected graph. If $\sigma_2(G)\geq 2\widetilde{\alpha}(G)+1$, then $G$ is hamiltonian-connected.
\end{theorem}
The $3$-connectivity condition in Theorem \ref{Theorem-ore-hamiltonian-connected} is necessary. Consider the graph $G = (K_{a-2}\cup K_1)\vee K_2$. Clearly, $\sigma_2(G) =a+1$ and $\widetilde{\alpha}(G)\leq 3$. Note that for $a\geq 6$, we have that $\sigma_2(G)\geq 7\geq 2\widetilde{\alpha}(G)+1.$ Since $\kappa(G)=2$, $G$ is not hamiltonian-connected.
 
We organize the remainder of this paper as follows:  Section \ref{Proof-hamiltonian-traceable} presents the proofs of Theorem \ref{Theorem-ore-hamiltonian} and Corollary \ref{Corollary-ore-traceable}, while Section \ref{Proof-hamiltonian-connected} focuses on the proof of Theorem \ref{Theorem-ore-hamiltonian-connected}.

\section{Proofs of Theorem \ref{Theorem-ore-hamiltonian} and Corollary \ref{Corollary-ore-traceable}}\label{Proof-hamiltonian-traceable}
The aim of this section is to prove Theorem \ref{Theorem-ore-hamiltonian} and Corollary \ref{Corollary-ore-traceable}. Before proceeding with the proof, we list  some notations and  observations that will be needed in later proofs. 
Let $P$ be an oriented $(u,v)$-path. For $x\in V(P)$ with $x\ne v$, denote by $x^+$ the immediate successor on $P$.  For $x\in V(P)$  with $x\ne u$, denote by $x^-$ the predecessor on $P$. 
For $S\subseteq V(P)$, let $S^+=\{x^+:x\in S\setminus\{v\}\}$ and $S^-=\{x^-:x\in S\setminus\{u\}\}$. Obviously, $|S^+|=|S|$ or $|S^+|=|S|-1$. 
For $x,y\in V(P)$,  $\overrightarrow{P}[x,y]$ denotes the segment of $P$ from $x$ to $y$ which follows the orientation of $P$, while $\overleftarrow{P}[x,y]$ denotes the opposite segment of $P$ from $y$ to $x$. Particularly, if $x=y$, then $\overrightarrow{P}[x,y]=\overleftarrow{P}[x,y]=x$.
\begin{observation}\label{Observation-hamiltonian}
Let $G$ be a traceable graph of order $n$, and let $P=v_1,v_2,\ldots,v_n$ be a Hamilton path of $G$. Then  $G$ is hamiltonian in any of the following three situations.
\begin{itemize}
    \item[$(1)$] There exists $i\in [2,\,n-1]$ such that $v_i\sim v_1$ and $v_{i-1}\sim v_n$.
    \item[$(2)$]  For an integer $k\in [2,\,n-1]$, there exist $i\in [2,\,k]$ and $j\in [k,\,n-1]$ such that $v_i\sim v_1 $, $v_j\sim v_n$ and $v_{i-1}\sim v_{j+1}$. Note that $i=j$ is possible here.
    \item[$(3)$] For an integer $k\in [2,\,n-1]$, there exist $i\in [k,\,n-1]$ and $j\in [1,\,k-1] $ such that $v_i\sim v_1 $, $v_j\sim v_n$ and $v_{i+1}\sim v_{j+1}$. Note that $i=j+1$ is possible here.
\end{itemize}
\end{observation}
Observation~\ref{Observation-hamiltonian}~$(1)$ yields a standard proof of Dirac's and Ore's theorems. 
 Observation~\ref{Observation-hamiltonian}~$(2)$ involves noncrossing edges from the endpoints, 
whereas Observation~\ref{Observation-hamiltonian}~$(3)$ involves crossing edges.

\begin{proof}[\bf Proof of Theorem \ref{Theorem-ore-hamiltonian}.]
We prove Theorem \ref{Theorem-ore-hamiltonian} by contradiction. Let $G$ be a counterexample to Theorem \ref{Theorem-ore-hamiltonian} with size maximum.  Then $G$ is not a complete graph, and hence $\widetilde{\alpha}(G)\geq 2$.  For any two nonadjacent vertices $u,v$
of $G$, let $G_{uv}$ be the graph obtained from $G$ by adding a new edge $uv$. Note that adding edges does not increase the bipartite-hole-number. Therefore, by the choice of $G$, we have that $G_{uv}$ is hamiltonian. Note that $G$ is non-hamiltonian. This implies that $uv$ is included in every Hamilton cycle of $G_{uv}$. Moreover, there is a Hamilton $(u,v)$-path in $G$. Now, let $P$ be a Hamilton $(u,v)$-path in $G$ and assume that $P$ is chosen such that  $\min\{d_G(u), d_G(v)\}$ is as large as possible. For convenience, assume that $P=v_1,v_2,\ldots,v_n$ with $v_1=u$ and $v_n=v$, and that $d_G(v_1)\leq d_G(v_n)$.

Let $s\in [t]$ satisfy $\widetilde{\alpha}(G)+1=s+t$, and  assume that $G$ has no  $(s,t)$-bipartite-hole. Since $\widetilde{\alpha}(G)\geq 2$, $1\leq s\leq \frac{\widetilde{\alpha}(G)+1}{2}<\widetilde{\alpha}(G)$.
We complete the proof of Theorem \ref{Theorem-ore-hamiltonian} by considering the following two cases.

\begin{case}
$d_G(v_1)\geq \widetilde{\alpha}(G)$.
\end{case}
Since  $1\leq s <\widetilde{\alpha}(G)$, there exists an integer $k\in [2,\,n-1]$ such that $|N_G(v_1)\cap \{v_i:~i\in [2,\,k]\}|=s$. Denote $S_1=N_G(v_1)\cap \{v_i:~i\in [2,\,k]\}$, $S_2=N_G(v_1)\cap \{v_i:~i\in [k+1,\,n-1]\}$, $T_1=N_G(v_n)\cap\{v_j:~j\in [k,\,n-1]\}$ and $T_2=N_G(v_n)\cap\{v_j:~j\in [2,\,k-1]\}$. Clearly, $N_G(v_1)$ is the disjoint union of $S_1$ and $S_2$, and $N_G(v_n)$ is the disjoint union of $T_1$ and $T_2$.

On the one hand, by Observation \ref{Observation-hamiltonian}~$(2)$, we have that
\begin{flalign*}
[S_1^-, T_1^+]=\emptyset.
\end{flalign*}
Note that  $d_G(v_n)\geq d_G(v_1) \geq \widetilde{\alpha}(G)$. Since $G$ has no $(s,t)$-bipartite-hole, $|T_1|\leq t-1$.  However, it follows that 
\begin{flalign*}
    |T_2|\geq d_G(v_n)-(t-1)\geq \widetilde{\alpha}(G)-(t-1)=s.
\end{flalign*}

On the other hand, by Observation \ref{Observation-hamiltonian}~$(1)$ and $(3)$, we have that 
\begin{flalign*}
    [S_2^+\cup \{v_1\},\, T_2^+]=\emptyset.
\end{flalign*}
Since $G$ has no $(s,t)$-bipartite-hole,
\begin{flalign*}
    |S_2^+|=|S_2^+\cup\{v_1\}|-1\leq (t-1)-1=t-2.
\end{flalign*}
This implies that 
\begin{flalign*}
\widetilde{\alpha}(G)\leq d_G(v_1)=|S_1|+|S_2|\leq s+t-2=\widetilde{\alpha}(G)-1,
\end{flalign*}
 a contradiction.
 \begin{case}
    $d_G(v_1)<\widetilde{\alpha}(G).$
 \end{case}
 Since $G$ is $2$-connected, $v_1$ has a neighbor distinct from $v_2$. Let $v_t$ be a neighbor of $v_1$, and choose $t$ to be as large as possible.  Now, $v_{t-1}, \overleftarrow{P}[v_{t-1}, v_1],v_1,v_t,\overrightarrow{P}[v_t,  v_n], v_n$  is a Hamilton path with endpoints $v_{t-1}$ and $v_n$.  By the choice of $P$, we have that $d_G(v_{t-1})\leq d_G(v_1)<\widetilde{\alpha}(G)$. Since $\sigma_2(G)\geq 2\widetilde{\alpha}(G)$, $v_{1}\sim v_{t-1}$. Repeating the process over and over again,
 it follows that $d_G(v_i)<\widetilde{\alpha}(G)$ and $v_i\nsim v_n$ for each $i\in [t-1]$. Since $\sigma_2(G)\geq 2\widetilde{\alpha}(G)$, $\{v_1,v_2,\ldots,v_{t-1} \}$ is a clique. Since $G$ is $2$-connected,
 \begin{flalign*}
    [\{v_1,v_2,\ldots,v_{t-1}\},\, \{v_{t+1},v_{t+2},\ldots,v_n\}]\neq \emptyset.
 \end{flalign*}
Therefore, there exist $j\in [2,\, t-1]$ and $j'\in [t+1,\,n-1]$ such that $v_j\sim v_{j'}$. Now,
\begin{flalign*}
    v_{j'-1},\overleftarrow{P}[v_{j'-1},  v_{j+1}],v_{j+1},v_1,\overrightarrow{P}[v_{1},  v_{j}],v_j,v_{j'}, \overrightarrow{P}[v_{j'},  v_n],v_n
\end{flalign*}
is a Hamilton path with endpoints $v_{j'-1}$ and $v_n$. By the choice of $P$, we have that 
\begin{flalign*}
 d_G(v_{j'-1})\leq d_G(v_1)<\widetilde{\alpha}(G).   
\end{flalign*}
Therefore, $\{v_1,v_2,\ldots,v_{t-1} \}\cup \{v_{j'-1}\}$ is a clique. This implies that $j'=t+1$. Now, 
\begin{flalign*}
    v_t,\overleftarrow{P}[v_t,v_{j+1}],v_{j+1},v_1,\overrightarrow{P}[v_1,v_{j}],v_j,v_{t+1},\overrightarrow{P}[v_{t+1},  v_n],v_n
\end{flalign*}
is a Hamilton path with endpoints $v_t$ and $v_n$. By the choice of $P$, we have that 
\begin{flalign*}
d_G(v_{t})\leq d_G(v_1)=t-1<t\leq d_G(v_{t}),
\end{flalign*}
 a contradiction. This completes the proof of Theorem \ref{Theorem-ore-hamiltonian}.
\end{proof}
For two graphs $G$ and $H,$ $G\vee H$ denotes the {\it join} of $G$ and $H,$ which is obtained from the disjoint union $G+H$ by adding edges joining every vertex of $G$ to every vertex of $H.$ The following trick is well-known (e.g., \cite{Chvatal1972}, p.112).
\begin{lemma}\label{Lemma-traceable-hamiltonian}
Let $G$ be a graph and denote $H=G\vee K_1$. Then $G$ is traceable if and only if $H$ is hamiltonian, and $\kappa(G)=k$ if and only if $\kappa(H)=k+1$.
\end{lemma}
\begin{proof}[\bf Proof of Corollary \ref{Corollary-ore-traceable}.]
   Let $ H = G \vee K_1 $. By the definition of the bipartite-hole-number, we have $ \widetilde{\alpha}(H) = \widetilde{\alpha}(G) $. It follows that

\[
\sigma_2(H) \geq \sigma_2(G) + 2 \geq 2 \widetilde{\alpha}(G) = 2 \widetilde{\alpha}(H).
\]

Applying Theorem~\ref{Theorem-ore-hamiltonian}, we conclude that $ H $ is hamiltonian. Therefore, by Lemma~\ref{Lemma-traceable-hamiltonian}, $ G $ is traceable.
\end{proof}

\section{Proof of Theorem \ref{Theorem-ore-hamiltonian-connected}}\label{Proof-hamiltonian-connected}

We prove Theorem \ref{Theorem-ore-hamiltonian-connected} by contradiction. Let $G$ be a counterexample to Theorem \ref{Theorem-ore-hamiltonian-connected} with size maximum.  Then $G$ is non-complete, and hence $\widetilde{\alpha}(G)\geq 2$.  For any two nonadjacent vertices $x,y$ of $G$, let $G_{xy}$ be the graph obtained from $G$ by adding a new edge $xy$. Note that adding edges does not increase the bipartite-hole-number. Therefore, by the choice of $G$, we have that $G_{xy}$ is hamiltonian-connected. 

Let $u,v\in V(G)$ be any two distinct vertices. We may assume that $G$ has no Hamilton $(u,v)$-path. For every edge $e\in E(\overline{G})$, the graph $G_e$ is hamiltonian-connected. This implies that every Hamilton $(u,v)$-path contains the edge $e$, and that $u\sim v$. Let $P$ be a Hamilton $(u,v)$-path in $G_e$ that contains the edge $e=xy$, and assume that $e$ is chosen such that $\min\{d_G(x), d_G(y)\}$ is as large as possible.

For convenience, assume that $P = v_1, v_2, \ldots, v_n$, where $v_1 = u$ and $v_n = v$, and the edge $e=v_r v_{r+1}$, with $d_G(v_{r+1}) \geq d_G(v_{r})$.

Let $s\in [t]$ satisfy $\widetilde{\alpha}(G)+1=s+t$, and  assume that $G$ has no  $(s,t)$-bipartite-hole. 
Since $\widetilde{\alpha}(G)\geq 2$, $1\leq s\leq \frac{\widetilde{\alpha}(G)+1}{2}<\widetilde{\alpha}(G)$.
We complete the proof of Theorem \ref{Theorem-ore-hamiltonian-connected} by considering the following two cases.

\begin{case}
$d_G(v_{r})\geq \widetilde{\alpha}(G)+1$.
\end{case}
Since $s< \widetilde{\alpha}(G)$, there exists an integer $k\in [1,\,n]$ such that $|N_G(v_r)\cap \{v_i:~i\in [1,\,k]\}|=s$. We further assume that $k$ is chosen to be the minimum. This implies that $v_rv_k\in E(G)$. Denote $S_1=N_G(v_r)\cap \{v_i:~i\in [1,\,k]\}$. That is, $|S_1|=s$. Since $v_rv_{r+1}\notin E(G)$, either $k\in [1,\,r-1]$ or $k\in [r+2,\,n]$. We distinguish two subcases.

\begin{subcase}
$k\in [1,\,r-1]$.
\end{subcase}

Denote $T_{1}=N_G(v_{r+1})\cap \{v_j:~j\in [k+1,\,r-1]\}$ and $R_1=N_G(v_{r+1})\cap \{v_j:~j\in [r+2,\,n]\}$.
Suppose $|T_1\cup R_1|\geq t$.  Then $|T_1^+\cup R_1^-|\geq t$. Since $G$   has no $(s,t)$-bipartite-hole, 
\begin{flalign*}
    [S_1^+,\, T_1^+\cup R_1^-]\neq \emptyset.
\end{flalign*}
If $[S_1^+,\, T_1^+]\neq \emptyset$, then there exist $v_j\in S_1$ and $v_{j'}\in T_1$ such that $v_j^+v_{j'}^+\in E(G)$. However, it follows that 
\begin{flalign*}
v_1,\overrightarrow{P}[v_1,v_j],v_j,v_r, \overleftarrow{P}[v_r,v_{j'}^+],v_{j'}^+,v_j^+,\overrightarrow{P}[v_j^+,v_{j'}],v_{j'},v_{r+1},\overrightarrow{P}[v_{r+1},v_n],v_n
\end{flalign*}
is a Hamilton $(u,v)$-path in $G$, a contradiction.

If $[S_1^+,\, R_1^-]\neq \emptyset$, then there exist $v_j\in S_1$ and $v_{j'}\in R_1$ such that $v_j^+v_{j'}^-\in E(G)$. However, it follows that 
\begin{flalign*}
    v_1,\overrightarrow{P}[v_1,v_j],v_j,v_r, \overleftarrow{P}[v_r,v_{j}^+],v_{j}^+,v_{j'}^-,\overleftarrow{P}[v_{j'}^-,v_{r+1}],v_{r+1},v_j',\overrightarrow{P}[v_j',v_n],v_n
\end{flalign*}
is a Hamilton $(u,v)$-path in $G$, a contradiction. Therefore, $|T_1\cup R_1|\leq t-1$. 

Denote $S_2=N_G(v_{r})\cap \{v_i:~i\in [k+1,\,r-1]\}$, $Z_2=N_G(v_r)\cap \{v_i:~i\in [r+2,n]\}$  and $T_2=N_G(v_{r+1})\cap \{v_j:~j\in [2,\,k]\}$. 

Since $d_G(v_{r+1}) \geq d_G(v_{r})\geq \widetilde{\alpha}(G)+1$, 
$$|T_2|\ge d_G(v_{r+1})-|T_1\cup R_1|-1\geq \widetilde{\alpha}(G)+1-(t-1)-1=s.$$
Suppose $|S_2\cup Z_2|\geq t$. Then $|S_2^+\cup Z_2^-|\geq t$. Since $G$ has no $(s,t)$-bipartite-hole, 
\begin{flalign*}
[S_2^+\cup Z_2^-, T_2^-]\neq \emptyset.
\end{flalign*}

If $[S_2^+,\, T_2^-]\neq \emptyset$, then there exist $v_j\in S_2$ and $v_{j'}\in T_2$ such that $v_j^+v_{j'}^-\in E(G)$. However, it follows that 
\begin{flalign*}
    v_1,\overrightarrow{P}[v_1,v_{j'}^-],v_{j'}^-,v_j^+,\overrightarrow{P}[v_j^+,v_r],v_r,v_j,\overleftarrow{P}[v_j,v_{j'}],v_{j'},v_{r+1},\overrightarrow{P}[v_{r+1},v_n],v_n
\end{flalign*}
is a Hamilton $(u,v)$-path in $G$, a contradiction.

If $[Z_2^-,\, T_2^-]\neq \emptyset$, then there exist $v_j\in Z_2$ and $v_{j'}\in T_2$ such that $v_j^-v_{j'}^-\in E(G)$. However, it follows that 
\begin{flalign*}
  v_1,\overrightarrow{P}[v_1,v_{j'}^-],v_{j'}^-,v_j^-,\overleftarrow{P}[v_j^-,v_{r+1}],v_{r+1},v_{j'},\overrightarrow{P}[v_{j'},v_r],v_r,v_j,\overrightarrow{P}[v_j,v_n],v_n
\end{flalign*}
is a Hamilton $(u,v)$-path in $G$, a contradiction.

Therefore, $|S_2\cup Z_2|\leq t-1$. However, it follows that 
\begin{flalign*}
d_G(v_{r})=|S_1|+|S_2\cup Z_2|\leq s+t-1<s+t=\widetilde{\alpha}(G)+1\leq d_G(v_{r}),
\end{flalign*}
 a contradiction.
\begin{subcase}
    $k\in [r+2,\, n]$.
\end{subcase}
Denote $S_3=N_G(v_r)\cap \{v_i:~i\in [k,\,n]\}$. Recall that $|S_1|=s$ and $v_rv_k\in E(G)$. Since $d_G(v_r)\ge \widetilde{\alpha}(G)+1$, we have 
$$|S_3|=d_G(v_r)+1-|S_1|\ge \widetilde{\alpha}(G)+2-s=t+1.$$
Then there exists an integer $k'\in [k,\,n]$ such that $|N_G(v_r)\cap \{v_i:~i\in [k',\,n]\}|=s+1$. We further assume that $k'$ is chosen to be the maximum. This implies that $v_rv_{k'}\in E(G)$. Furthermore, denote $Z_3=N_G(v_r)\cap \{v_i:~i\in [k',\,n-1]\}$, $T_3=N_G(v_{r+1})\cap \{v_j:~j\in [r+2,k'-1]\}$ and $R_3=N_G(v_{r+1})\cap \{v_j:~j\in [1,r-1]\}$.

Suppose $|T_3\cup R_3|\geq t$. Then $|T_3^-\cup R_3^+|\geq t$. Since $G$ has no $(s,t)$-bipartite-hole,
\begin{flalign*}
    [Z_3^+,\,T_3^-\cup R_3^+]\neq \emptyset.
\end{flalign*}

If $[Z_3^+,\,T_3^-]\neq \emptyset$, then there exist $v_j\in Z_3$ and $v_{j'}\in T_3$ such that $v_j^+v_{j'}^-\in E(G)$. However, it follows that 
\begin{flalign*}
  v_1,\overrightarrow{P}[v_1,v_{r}],v_r,v_j, \overleftarrow{P}[v_j,v_{j'}],v_{j'},v_{r+1},\overrightarrow{P}[v_{r+1},v_{j'}^-], v_{j'}^-,v_j^+, \overrightarrow{P}[v_j^+,v_{n}],v_n
\end{flalign*}
is a Hamilton $(u,v)$-path in $G$, a contradiction.    

If $[Z_3^+,\,R_3^+]$, then there exist $v_j\in Z_3$ and $v_{j'}\in R_3$ such that $v_j^+v_{j'}^+\in E(G)$. However, it follows that 
\begin{flalign*}
v_1,\overrightarrow{P}[v_1,v_{j'}],v_{j'},v_{r+1},\overrightarrow{P}[v_{r+1},v_j],v_j,v_r,\overleftarrow{P}[v_r,v_{j'}^+],v_{j'}^+,v_j^+,\overrightarrow{P}[v_j^+,v_n],v_n
\end{flalign*}
is a Hamilton $(u,v)$-path in $G$, a contradiction. Therefore, $|T_3\cup R_3|\leq t-1$.

Denote $T_4=N_G(v_{r+1})\cap \{v_j:~j\in [k',\,n]\}$. Since $d_G(v_{r+1}) \geq d_G(v_{r})\geq \widetilde{\alpha}(G)+1$,
\begin{flalign*}
    |T_4|=d_G(v_{r+1})-|T_3\cup R_3|\geq \widetilde{\alpha}(G)+1-(t-1)=s+1.
\end{flalign*}
Then there exists an integer $k^*\in [k',\,n]$ such that $|N_G(v_{r+1})\cap \{v_j:~j\in [k^*,\,n]\}|=s$. We further assume that $k^*$ is chosen to be the maximum. This implies that $v_{r+1}v_{k^*}\in E(G)$ and $k^*\geq k'+1$.

Denote $R_4=N_G(v_{r+1})\cap \{v_j:~j\in [k^*,\,n]\}$, $S_4=N_G(v_r)\cap \{v_i:~i\in [1,\,r-1]\}$ and $Z_4=N_G(v_r)\cap \{v_i:~i\in [r+2,\,k']\}$. Suppose $|S_4\cup Z_4|\geq t$. Then $|S_4^+\cup Z_4^-|\geq t$. Since $G$ has no $(s,t)$-bipartite-hole,
\begin{flalign*}
[R_4^-,\, S_4^+\cup Z_4^-]\neq \emptyset. 
\end{flalign*}
If $[R_4^-,\, S_4^+]\neq \emptyset$, then there exist $v_j\in R_4^-,$ and $v_{j'}\in  S_4^+$ such that $v_j^-v_{j'}^+\in E(G)$. However, it follows that 
\begin{flalign*}
v_1,\overrightarrow{P}[v_1,v_{j'}],v_{j'},v_r,\overleftarrow{P}[v_r,v_{j'}^+],v_{j'}^+,v_{j}^-,\overleftarrow{P}[v_{j}^-,v_{r+1}],v_{r+1},v_j,\overrightarrow{P}[v_j,v_n],v_n
\end{flalign*}
is a Hamilton $(u,v)$-path in $G$, a contradiction.

If $[R_4^-,\, Z_4^-]\neq \emptyset$, then there exist $v_j\in R_4^-,$ and $v_{j'}\in  Z_4^-$ such that $v_j^-v_{j'}^-\in E(G)$. However, it follows that 
\begin{flalign*}
v_1,\overrightarrow{P}[v_1,v_{r}],v_r,v_{j'},\overrightarrow{P}[v_{j'},v_{j}^-],v_j^-,v_{j'}^-,\overleftarrow{P}[v_{j'}^-,v_{r+1}],v_{r+1},v_j,\overrightarrow{P}[v_j,v_n],v_n
\end{flalign*}
is a Hamilton $(u,v)$-path in $G$, a contradiction. Therefore, $|S_4\cup Z_4|\leq t-1$. This implies that 
\begin{flalign*}
d_G(v_{r})=|S_4|+|Z_4|+|N_G(v_r)\cap \{v_i:~i\in [k',\,n-1]\}|-1\leq t-1+s<\widetilde{\alpha}(G)+1\leq d_G(v_{r}),
\end{flalign*}
a contradiction.
\begin{case}
$d_G(v_r)\leq \widetilde{\alpha}(G)$.
\end{case}
Denote $X=\{v_i:~d_G(v_i)\leq \widetilde{\alpha}(G)\}$ and $Y=\{v_j:~d_G(v_j)\geq \widetilde{\alpha}(G)+1\}$. By the choice of the edge $e$ and $\sigma_2(G)\geq 2\widetilde{\alpha}(G)+1$, we have that $X$ is a clique and $Y$ is a clique. Since $G$ is $3$-connected, there exist three vertex disjoint $(X,Y)$-paths $P_1,P_2,P_3$. If $|X|\ge 3$ and $|Y|\ge 3$, then $\{P_1,P_2,P_3\}$ is a matching of cardinality $3$. If $|X|\le 2$ or $|Y|\le 2$, combining the fact that $G$ is $3$-connected, then $|N_{G[Y]}(X)|\ge 3$ or $|N_{G[X]}(Y)|\ge 3$.  
In any case, $G$ is hamiltonian-connected.
This completes the proof of Theorem \ref{Theorem-ore-hamiltonian-connected}.\hfill$\Box$



\section*{Acknowledgement}  The authors are grateful to Professor Xingzhi Zhan for his constant support and guidance. This research  was supported by the NSFC grant 12271170 and Science and Technology Commission of Shanghai Municipality (STCSM) grant 22DZ2229014.

\section*{Declaration}

	\noindent$\textbf{Conflict~of~interest}$
	The authors declare that they have no known competing financial interests or personal relationships that could have appeared to influence the work reported in this paper.
	
	\noindent$\textbf{Data~availability}$
	Data sharing not applicable to this paper as no datasets were generated or analysed during the current study.

\end{document}